\def\C{\mathbb{C}}
\def\R{\mathbb{R}}
\def\Z{\mathbb{Z}}
\def\N{\mathbb{N}}
\def\Q{\mathbb{Q}}
\def\A{\mathcal{A}}
\def\B{\mathcal{B}}
\def\Zbb{\mathbb Z[\beta,\beta^{-1}]}
\def\pf{\begin{proof}}
\def\pfk{\end{proof}}

\documentclass{article}
\usepackage{amssymb,amsmath,amsthm}
\usepackage{algorithm,algorithmic}
\usepackage{color,epic,eepic}

\newtheorem{lem}{Lemma}
\newtheorem{thm}[lem]{Theorem}
\newtheorem{prop}[lem]{Proposition}
\newtheorem{coro}[lem]{Corollary}
\newtheorem{definition}[lem]{Definition}
\newtheorem{example}[lem]{Example}
\newtheorem{pozn}[lem]{Remark}

\begin{document}
\title{On periodic representations in non-Pisot bases}
\author{S. Baker$^1$, Z. Mas\'akov\'a$^2$, E. Pelantov\'a$^2$, T. V\'avra$^2$\\[2mm]
 \emph{$^1$ Department of Mathematics and Statistics, University of Reading}\\
\emph{Whiteknights, PO Box 220, Reading, RG6 6AX, UK}\\[2mm]
\emph{$^2$ Department of Mathematics FNSPE, Czech Technical University in Prague}\\
\emph{Trojanova 13, 120 00 Praha 2, Czech Republic}\\
}

\maketitle

\begin{abstract}
We study periodic expansions in positional number systems with a base $\beta\in\C,\ |\beta|>1$, and with coefficients in a finite set of digits $\A\subset\C.$ We are interested in determining those algebraic bases for which there exists $\A\subset \Q(\beta),$ such that all elements of $\Q(\beta)$ admit at least one eventually periodic representation with digits in $\A$. In this paper we prove a general result that guarantees the existence of such an $\A$. This result implies the existence of such an $\A$ when $\beta$ is a rational number or an algebraic integer with no conjugates of modulus $1$.

We also consider eventually periodic representations of elements of $\Q(\beta)$ for which the maximal power of the representation is proportional to the absolute value of the represented number, up to some universal constant. We prove that if every element of $\Q(\beta)$ admits such a representation then $\beta$ must be a Pisot number or a Salem number. This result generalises a well known result of Schmidt \cite{Schmidt}.
\end{abstract}
\section{Introduction}

We consider representations of numbers in a base $\beta\in\C$, $|\beta|>1$, using a finite alphabet of digits. When $\beta$ is positive and real, the classical way to obtain representations is to apply the greedy algorithm and use the `canonical' alphabet $\{a\in\Z: 0\leq a<\beta\}$. A well known result of Schmidt~\cite{Schmidt} states that if $\beta$ is a Pisot number then every positive element of the field $\Q(\beta)$ has eventually periodic greedy $\beta$-expansion, i.e.
$$
{\rm Per}(\beta):=\{x\geq 0: x\text{ has eventually periodic greedy $\beta$-expansion}\}=\Q(\beta)\cap[0,+\infty)\,.
$$
On the other hand, apart from Pisot numbers, only Salem numbers may have this property. However, there are no examples of Salem numbers of this type.




Alternative ways of representing numbers in a real base (balanced system with symmetric alphabet, negative base, etc.) were considered in~\cite{AkiSch},~\cite{KaSt}, and \cite{ItoSadahiro}. Corresponding analogues of the Schmidt's result were shown to hold~\cite{FrLai} and ~\cite{DoMaVa}. This emphasises the importance of Pisot numbers in the theory of number representations.

It is natural to wonder what happens if we relax the condition that the expansion in base $\beta$ is generated by a specific algorithm and allow representations in a general alphabet $\A$. We shall study the so called $(\beta,\A)$-representations for $\beta\in\C$, $|\beta|>1$ and $\A\subset\C$ finite, i.e. expressions of the form $\sum_{k\geq -L}{a_k}\beta^{-k},\ a_k\in\A$. In this article we focus on the set
$$
{\rm Per}_\A(\beta) = \big\{x\in\C : x\text{ has an eventually periodic $(\beta,\A)$-representation}\big\}
$$
and study the following question:

\smallskip
{\it Question:} Which bases $\beta$ admit an alphabet $\A$ for which ${\rm Per}_\A(\beta)=\Q(\beta)$?
\smallskip

Clearly, from Schmidt, any Pisot number $\beta$ satisfies this property choosing $\A=\{a\in\Z:-\beta<a<\beta\}$.
As we will show, it is satisfied also by complex Pisot bases and a broad class of non-Pisot and non-Salem numbers, in particular, all algebraic integers without a conjugate on the unit circle, see Corollary~\ref{coro:algint}. Surprisingly, being an algebraic integer is not a necessary condition. For example, rational numbers have this property (Corollary~\ref{coro:rational}). In Theorem~\ref{t:X}, we provide a sufficient condition on the base $\beta$ so that there exists an alphabet $\A$ allowing eventually periodic representations of every element of $\Q(\beta)$.


Although alphabets exist which yield eventually periodic representations for every element of $\Q(\beta)$ for very general bases, perhaps all algebraic $\beta$, Pisot and Salem numbers still play a crucial role in the study of our question. In fact, they are the only bases for which eventually periodic representations can satisfy a special property. Roughly speaking, the represented number is in its modulus proportional to the highest power of the base used, and moreover, the proportionality is uniform for every element of the field $\Q(\beta)$. We call such representations `weak greedy', these ideas are expressed formally in Definition~\ref{d:weakgreedy}. This property of Pisot and Salem numbers may be seen as a strengthening of Schmidt's theorem; we state it as Theorem~\ref{t:Simon}.

%
%
%
%
%
%

\section{Representations of numbers}

An algebraic number $\alpha$ is a zero of a polynomial $f(x)\in\Z[x].$ The nonzero polynomial $m(x)=a_dx^d+\dots+a_1x+a_0\in\Z[x]$ of minimal degree such that $m(\alpha)=0$ and $\gcd(a_d,\dots,a_0)=1$ is called the minimal polynomial of $\alpha.$ If the leading coefficient of the minimal polynomial satisfies $a_d=\pm1$, then $\alpha$ is called an algebraic integer. Galois conjugates of $\alpha$ are the roots of the minimal polynomial of $\alpha$.

A very important role in the theory of numeration systems is played by Pisot and Salem numbers.
\begin{definition}
\begin{enumerate}
	\item A Pisot number is an algebraic integer $\beta>1$ whose Galois conjugates are $<1$ in absolute value.

	\item A complex Pisot number is an algebraic integer $\beta\in\C\setminus\R,\ |\beta|>1$ whose Galois conjugates except for the complex conjugate are $<1$ in absolute value.

	\item A  Salem number is an algebraic integer $\beta>1$ whose Galois conjugates are $\leq1$ in absolute value and at least on of them is equal to $1$ in absolute value.
\end{enumerate}
\end{definition}
\begin{definition}
Let $\beta\in\C$, $|\beta|>1$, and let $\A\subset\C$ be a finite set containing $0$.
An expression
\begin{equation}\label{eq:1}
x=\sum_{k\geq -L}a_k\beta^{-k}, \quad a_k\in \A
\end{equation}
is called a $(\beta,\A)$-representation of $x$. If $a_{-L}\neq 0$, then $L$ is called the leading index of the
$(\beta,\A)$-representation.
\end{definition}
\noindent
Suppose we are given a set $\Omega$ and a mapping $D:\Omega\to\A$ such that for every $x\in \Omega$
\begin{equation}\label{eq:2}
T(x):=\beta x-D(x)\in \Omega\,.
\end{equation}
Then the map $x\mapsto T(x)$ is a transformation $T:\Omega\to \Omega$ which constructs a $(\beta,\A)$-representation
for every $x\in \Omega.$ This representation is of the form
\begin{equation}\label{eq:3}
x=\sum_{k\geq 1}a_k\beta^{-k}, \quad a_k\in \A\,.
\end{equation}
It is a simple observation that the coefficients $a_k$ are determined by the equation
$$
a_k=D\big(T^{k-1}(x)\big)\,.
$$
Moreover, every $x\in \bigcup_{n\in\N}\beta^n \Omega$ has a $(\beta,\A)$-representation of the form~\eqref{eq:1}.

\begin{example}\label{ex:1}
Let us give several known examples of numeration systems based on the construction described above.
\begin{enumerate}
  \item For a positive real base $\beta>1$, the greedy expansion of any $x\in[0,+\infty)$  was considered by R\'enyi~\cite{renyi}.
It is given in the above framework with $\Omega=[0,1)$, $D(x)=\lfloor\beta x\rfloor,$ and the alphabet of digits
$\A=\{a\in\Z:0\leq a<\beta\}$.

  \item Akiyama and Scheicher~\cite{AkiSch} defined a representation for any real $x$
in the balanced system with $\Omega=[-\frac12,\frac12)$, $D(x)=\lfloor\beta x+\frac12\rfloor$ and the alphabet
$\A=\{a\in\Z:-\beta/2\leq a<\beta/2\}$.

  \item When the base is negative, say $-\beta<-1$, to obtain expansions for any real $x$ with alphabet $\A=\{a\in\Z:0\leq a\leq\beta\}$, one can use the transformation $T$ on $\Omega=[l,l+1)$ with $D= \lfloor-\beta x-l\rfloor$ and $l=-\beta/(\beta+1)$. This method was introduced by Ito and Sadahiro~\cite{ItoSadahiro}.
\end{enumerate}
\end{example}

Currently, there is no analogue of the greedy algorithm that can be successfully applied to any $\beta$.
$(\beta,\A)$-representation for special complex bases are studied in~\cite{HFI} and~\cite{KomLor}.
Nevertheless, Dar\'oczy and K\'atai~\cite{DarKat}, and then Thurston~\cite{Thurston}, proved that for any
non-real $\beta\in\C$ of modulus greater than 1, there exists a finite alphabet $\A\subset\C$ such that every $x\in \C$
has a $(\beta,\A)$-representation.

\begin{thm}{\cite{Thurston}}\label{t:Thurston}
Given a base $\beta\in\C$, $|\beta|>1$, and a finite set $\A\subset\C$. Let $\Omega$ be a bounded subset of $\C$ such that
\begin{itemize}
\item 0 belongs to the interior of $\Omega$;
\item $\beta \Omega \subset \A+\Omega = \bigcup_{a\in\A}(a+\Omega)$.
\end{itemize}
Then every $x\in\C$ has a $(\beta,\A)$-representation of the form~\eqref{eq:1}.
\end{thm}

\begin{pozn}
Theorem~\ref{t:Thurston} provides a formula for the mapping $D:\Omega\to \A$. Set $D(x)=a$ if $\beta x\in a+\Omega$. Such $a\in\A$ may not be unique, therefore we have a choice in how we define $D$.
\end{pozn}

\begin{pozn}\label{pozn:2}
Given a base $\beta\in\C$, $|\beta|>1$, one can easily find an alphabet $\A$ and a set $\Omega\subset\C$ satisfying the assumptions
of Theorem~\ref{t:Thurston}. For example, consider $\Omega=B(0,1)$ and $\A=\{x+iy: x,y\in\Z,\, B(x+iy,1)\cap B(0,|\beta|)\neq\emptyset\}$, where $B(z,r)$ stands for an open ball of radius $r$ centered at $z$. In~\cite{BrFrPeSv} it is shown that the alphabet can always be chosen as a symmetric subset of integers, $\A=\{-M,\dots,0,1,\dots,M\}\subset\Z$. The corresponding $\Omega\subset\C$ is shown in Figure~\ref{f}.
\end{pozn}

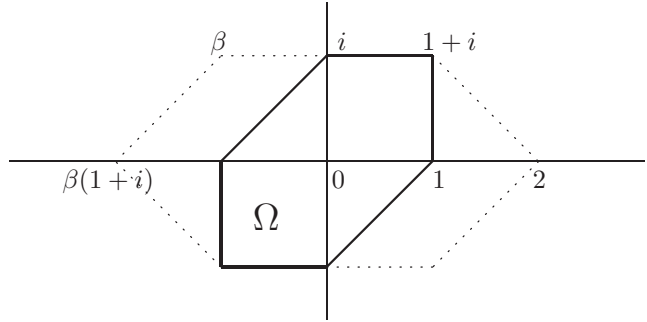
\begin{figure}[h]
\begin{center}
\setlength{\unitlength}{4pt}
\begin{picture}(60,30)
\thinlines
\put(0,15){\line(1,0){60}}
\put(30,0){\line(0,1){30}}
{\thicklines
\put(40,15){\line(0,1){10}}
\put(30,25){\line(1,0){10}}
\put(20,5){\line(1,0){10}}
\put(20,5){\line(0,1){10}}
\put(20,15){\line(1,1){10}}
\put(30,5){\line(1,1){10}}}
\put(5,12.5){$\beta(1+i)$}
\put(39,25.5){$1+i$}
\put(31,25.5){$i$}
\put(19,25.5){$\beta$}
\put(40,12.5){$1$}
\put(23,8.5){\Large $\Omega$}
\put(49.5,12.5){$2$}
\put(30.5,12.5){$0$}
\dottedline{1}(10,15)(20,25)(40,25)(50,15)(40,5)(20,5)(10,15)
\end{picture}
\end{center}
\caption{The region $\Omega$ for the base $\beta=i-1$ and the alphabet $\A=\{-1,0,1\}$ is the convex hull of points $\pm2,1\pm i,-1\pm i$. $\beta\Omega$ is marked by dotted lines, it can obviously by covered by the union of $\Omega$, $\Omega+1$ and $\Omega-1$.}
\label{f}
\end{figure}

\begin{pozn}\label{pozn:3}
Representations considered by R\'enyi, Akiyama and Scheicher, Ito and Sadahiro, or Thurston, all share the following feature: The leading index $L$ of the $(\beta,\A)$-representation of a number $x$ is closely related to the modulus of $|x|$. More precisely, there exists a constant $c>0$ such that for every $x\neq0$, one has
\begin{equation}\label{eq:*}
|x|\geq c|\beta|^L\,.
\end{equation}
For example, consider the $(\beta,\A)$-representation of a non-zero $x\in\C$ obtained from Theorem~\ref{t:Thurston}. Its leading index $L$ is found as the maximal integer such that $\frac{x}{\beta^{L-1}}\notin \Omega$. Then we have
$$
\frac{|x|}{|\beta|^{L-1}} \geq \inf\{|z|:z\in\C,\,z\notin \Omega\}>0\,.
$$
The last inequality follows from the assumption that 0 is an interior point of $\Omega$. This gives the estimate~\eqref{eq:*}.
\end{pozn}

\section{Eventually periodic representations in Pisot bases}

As we already mentioned, the question of the existence of an alphabet that allows $\Q(\beta)$ to have periodic expansions is already solved for real bases $\beta$ such that $|\beta|$ is a Pisot number, as a direct consequence of Schmidt's result.


For negative base $-\beta<-1$ described in Example~\ref{ex:1}, an analogous result to the one of Schmidt was proved by Frougny and Lai in \cite{FrLai}. In this case, if $\beta$ is Pisot, then all the elements of $\Q(\beta)$ admit periodic expansions with the alphabet $\{0,1,\dots,\lfloor\beta\rfloor\}.$


Let us focus on a general case. The equality ${\rm Per}_\A(\beta)=\Q(\beta)$ automatically implies that $\A\subset\Q(\beta)$. We will now show that if any element of $\Q(\beta)$ can be represented by an eventually periodic sequence of digits from $\A\subset\Q(\beta)$, then the same is true for a finite alphabet $\A'\subset\Z$.

\begin{lem}\label{l:integer alphabet}
Fix $\beta\in \C$, $|\beta|>1$.
\begin{enumerate}
  \item Let $\A\subset\Q(\beta)$ be such that every element of $\Q(\beta)$ has $(\beta,\A)$-representation. Then there exists an alphabet $\A'\subset \Z$ such that every element of $\Q(\beta)$ has a $(\beta,\A')$-representation.
  \item Let $\A\subset\Q(\beta)$ be such that every element of $\Q(\beta)$ has an eventually periodic $(\beta,\A)$-representation. Then there exists an alphabet $\A'\subset \Z$ such that every element of $\Q(\beta)$ has an eventually periodic $(\beta,\A')$-representation.
  \item Suppose there exists $c>0$ such that every $x\in \Q(\beta)$ has a $(\beta,\A)$-representation for which $|x|\geq c\cdot \beta^{L}$ where $L$ is the leading index. Then there exists $\A'\subset\Z$, $c'>0$ such that every $x\in \Q(\beta)$ has a $(\beta,\A')$-representation for which $|x|\geq c'\cdot \beta^{L'}$ where $L'$ is the leading digit.
\end{enumerate}
\end{lem}

\pf
We begin by proving item $1$. Let $\A=\{\epsilon_{j}\}_{j=1}^{N}$ be as in the statement of our lemma. Since each $\epsilon_{j}$ is an element of $\Q(\beta)$ we may express each $\epsilon_{j}$ in the following form
\begin{equation}
\label{Expand epsilon}
\epsilon_{j}=\frac{p_{0}^{(j)}+p_{1}^{(j)}\beta+\cdots +p_{d-1}^{(j)}\beta^{d-1}}{Q},
\end{equation}where each $p_{i}^{(j)}\in\mathbb{Z},$ $Q\in\N$, and $d\in\mathbb{N}$. Importantly $d$ and $Q$ do not depend on $\epsilon_{j}$. Let us take $x\in \Q(\beta)$ arbitrarily, by our assumption there exists a sequence $(a_{k})_{k=-L}^{\infty}$ with digits in $\A$ such that
\begin{equation}
\label{representation1}
x=\sum_{k\geq -L}a_{k}\beta^{-k}.
\end{equation}Substituting (\ref{Expand epsilon}) into (\ref{representation1}) we obtain
\begin{equation}
\label{representation2}
x=\sum_{k\geq-L}\Big(\frac{p_{0}^{(j(k))}+p_{1}^{(j(k))}\beta+\cdots +p_{d-1}^{(j(k))}\beta^{d-1}}{Q}\Big)\beta^{-k}.
\end{equation}Multiplying out the bracket in (\ref{representation2}) we see that $x$ can be expressed as
\begin{equation}
\label{representation3}
x=\sum_{k\geq -L-d+1}\Big(\frac{b_{k}}{Q}\Big)\beta^{-k},
\end{equation}where each $b_{k}\in\mathbb{Z}$ is the sum of at most $d$ elements from the set $\{p_{i}^{(j)}\}_{i,j}.$  Since each $b_{k}$ is is the sum of at most $d$ elements from the set $\{p_{i}^{(j)}\}_{i,j},$ we may conclude that there exists $M\in\mathbb{N}$ such that $|b_{k}|\leq M$ for all $k$. Importantly we may choose $M$ so that it does not depend on $x$. Thus by \eqref{representation3} every element of $\Q(\beta)$ has a $(\beta,\{\frac{-M}{Q}, \frac{-M+1}{Q},\ldots, \frac{M-1}{Q},\frac{M}{Q}\})$-representation.

Now let us take $y\in \Q(\beta)$ arbitrarily and consider $y/Q$. Since $y/Q\in\Q(\beta),$ by the above there exists a sequence $(c_{k})_{k\geq -L}$ with digits in $\{-M,-M+1,\ldots,M-1,M\}$ such that $$\frac{y}{Q}=\sum_{k\geq-L}\Big(\frac{c_{k}}{Q}\Big)\beta^{-k}.$$ Multiplying through by $Q$ yields $y=\sum_{k\geq-L}c_{k}\beta^{-k}$. Thus $y$ has a $(\beta,\{-M,-M+1,\ldots,M-1,M\})$-representation. Since $y$ was arbitrary and $M$ does not depend on $y$ item $1$ holds.

To prove item $2$ we remark that in \eqref{representation3} the quantity $b_{k}$ depends only upon $a_{k},\ldots,a_{k-d+1}.$ Therefore, if $(a_{k})_{k=-L}^{\infty}$ is eventually periodic so is $(b_{k})_{k=-L+d-1}^{\infty}.$ Thus we conclude that every $x\in\Q(\beta)$ has an eventually periodic $(\beta,\{\frac{-M}{Q}, \frac{-M+1}{Q},\ldots, \frac{M-1}{Q},\frac{M}{Q}\})$-representation if every $x\in \Q(\beta)$ has an eventually periodic $(\beta,\A)$-representation. If $x/Q$ has an eventually periodic $(\beta,\frac{-M+1}{Q},\ldots, \frac{M-1}{Q},\frac{M}{Q}\})$-representation then $x$ has an eventually periodic $(\beta,\{-M,-M+1,\ldots,M-1,M\})$-representation. Since $x$ is arbitrary item $2$ of our lemma holds.

To prove item $3$ we begin by fixing $x\in \mathbb{Q}(\beta)$. Consider $x/Q$, suppose it has a $(\beta,\A)$-representation with leading digit $L$ which satisfies $|x/Q|\geq c\cdot \beta^{L}$. By \eqref{representation3} we see that $x/Q$ has a $(\beta,\frac{-M+1}{Q},\ldots, \frac{M-1}{Q},\frac{M}{Q}\})$-representation with leading digit $L'.$ Importantly $L'\leq L+d-1$. Multiplying through by $Q$ we see that $x$ has a $(\beta,\{-M,-M+1,\ldots,M-1,M\})$-representation with leading digit $L'$. Using the inequalities $|x/Q|\geq c\cdot \beta^{L}$ and $L'\leq L+d-1$ we obtain $|x|\geq c'\cdot \beta^{L'}$ where $c'=Qc\beta^{-d+1}.$ The parameter $c'$ has no dependence on $x$ so we may conclude item $3$ from our lemma.

\pfk

Lemma \ref{l:integer alphabet} will often allow us to assume, without loss of generality, that an alphabet $\A$ satisfies $\A\subset\Z$.

First we show that negative Pisot and complex Pisot numbers admit an alphabet such that every element of the field $\Q(\beta)$ can be represented by an eventually periodic sequence. This is a generalization of one of the implications of the theorem of Schmidt~\cite{Schmidt}.

\begin{lem}\label{l:2}
Let $\beta$ or $-\beta$ be a Pisot number, or let $\beta$ be a complex Pisot number. Let $\A\subset\Q(\beta)$ be a finite alphabet of digits. If an $x\in \Q(\beta)$ has a $(\beta,\A)$-representation with leading index $L$, then it has an eventually periodic $(\beta,\A)$-representation with leading index $L$.
\end{lem}

\pf
Let $x\in\Q(\beta)$ and let $x=\sum_{k\geq -L}a_k\beta^{-k}$, $a_{-L}\neq 0$. For $n\geq -L$, define
$$
T_n:=\beta^n(x-\sum_{k=-L}^na_k\beta^{-k}) = \frac{a_{n+1}}{\beta}+\frac{a_{n+2}}{\beta^2}+\cdots\in\Q(\beta)\,.
$$
Since
\begin{equation}\label{eq:4}
\big|\frac{a_{n+1}}{\beta}+\frac{a_{n+2}}{\beta^2}+\cdots\big|\leq \max\{a:a\in\A\} \sum_{j=1}^\infty\frac1{|\beta|^j}\,,
\end{equation}
the sequence $(T_n)_{n\geq -L}$ is bounded.
Let $\beta'$ be a conjugate of $\beta$ and let $\sigma$ be the field isomorphism induced by $\beta\mapsto \sigma(\beta)=\beta'$. If $\beta$ is complex and $\beta'=\overline{\beta}$, then $\big(\sigma(T_n)\big)_{n\geq -L}$ is bounded due to~\eqref{eq:4}. For all other conjugates of $\beta$, consider
$$
\begin{aligned}
\big|\sigma(T_n)\big| &= \big|\sigma(\beta^nx - a_n - a_{n-1}\beta - \cdots - a_{-L}\beta^{n-L})\big| \leq \\
&\leq |\beta'|^n|\sigma(x)| + \max\{|a|:a\in\A\}\cdot \sum_{j=0}^\infty{|\beta'|^j}\,.
\end{aligned}
$$
As $|\beta'|<1$, the sequence $\big(\sigma(T_n)\big)_{n\geq -L}$ is bounded for any isomorphism $\sigma$.

Recall that $\beta$ is an algebraic integer, and thus for any $H>0$ the set
$$
\{z\in\Z[\beta]:|\sigma(z)|<H \text{ for every isomorphism }\sigma\}
$$
is finite.
Since $x\in\Q(\beta)$ and $\A\subset \Q(\beta)$, we can find $Q\in\N$ such that $x\in\frac1Q\Z[\beta]$ and $\A\subset\frac1Q\Z[\beta]$ .
Obviously, also $Q T_n\in\Z[\beta]$ for every $n\geq -L$. Therefore the elements of the sequence $(T_n)_{n\geq -L}$ take only finitely many values. In particular, there exist indices $n<n+l$ such that $T_n=T_{n+l}$. Thus
$$
T_n = \frac{a_{n+1}}\beta +\dots+\frac{a_{n+l}}{\beta^l}+\frac1{\beta^l}T_n.
$$
Thus we have obtained an eventually periodic representation with the period $a_{n+1}\dots a_{n+l}.
$
\pfk

We have seen in Remark~\ref{pozn:2} that for any complex base $\beta$ with $|\beta|>1$, one can find an alphabet so that every complex number can be represented. Moreover, by Remark~\ref{pozn:3}, the $(\beta,\A)$-representations satisfy property~\eqref{eq:*}.
Therefore, by Lemma~\ref{l:2}, if $\beta$ is a complex Pisot number, all elements of the field $\Q(\beta)$ have eventually periodic $(\beta,\A)$-representations satisfying~\eqref{eq:*}.
This leads us to a partial answer to our original question.

\begin{thm}\label{t:2}
For a base $\beta$, which is Pisot, complex Pisot, or the negative of a Pisot number, one can find an alphabet $\A\subset\Z$ so that every $x\in\Q(\beta)$ has eventually periodic $(\beta,\A)$-representation. 
\end{thm}

However, Pisot and complex Pisot bases are not the only bases for which an eventually periodic $(\beta,\A)$-representation of all elements of $\Q(\beta)$ can be found.

\begin{example}\label{ex:sqrt5}
Let $\beta=\sqrt5$, $\A=\{-2,-1,0,1,2\}$. Every $x\in\Q(\beta)$ can be written in the form $x=x_1+x_2\sqrt5$, where $x_1,x_2\in\Q$. Numbers $x_1,x_2$ have eventually periodic expansions in base $\gamma=5$ and alphabet $\A$, say
$x_1=\sum_{k\geq -L_1}b_k5^{-k}$, $x_2=\sum_{k\geq -L_2}c_k5^{-k}$. Together, we can write
$$
x=\sum_{k\geq -L_1} b_k\beta^{-2k} + \sum_{k\geq -L_2} c_k\beta^{-2k+1}\,.
$$
The resulting representation of $x$ in base $\beta$ is also eventually periodic.
\end{example}

In the same way, one can construct eventually periodic representations of numbers in $\Q(\beta)$ in systems where $\beta$ is any root of any Pisot or complex Pisot number.

\begin{coro}
Let $\gamma\in\C$, $|\gamma|>1$, and $\A\subset\Z$ be finite such that every $x\in\Q(\gamma)$ has an eventually periodic
$(\gamma,\A)$-representation. Let $m\in\N$ and $\beta$ be such that $\gamma=\beta^m$. Then any $x\in\Q(\beta)$ has an eventually periodic
$(\beta,\A)$-representation.
\end{coro}

\pf
Obviously $\Q(\gamma)$ is a subfield of $\Q(\beta)$ of dimension $d=[\Q(\beta):\Q(\gamma)]$. Then $1,\beta,\dots,\beta^{d-1}$ forms a basis of $\Q(\beta)$ over $\Q(\gamma)$. Every $x\in\Q(\beta)$ can thus be written
as $x=\sum_{i=0}^{d-1}x_i\beta^{i}$, where $x_i\in\Q(\gamma)$. By the assumption, every $x_i$ has an eventually periodic
$(\gamma,\A)$-representation. Obviously, one can write these representations in such a way that the period length is common for every $x_i$. The eventually periodic $(\beta,\A)$-representation of $x$ can be found realizing that every power of $\gamma$ can be written as $\gamma^k=\beta^{mk}$.
\pfk

\section{Weak greedy expansions}

In Example~\ref{ex:sqrt5} the representations of elements in $\Q(\sqrt5)$ were not obtained by any transformation of the type~\eqref{eq:2}. It is natural to ask whether these representations also satisfy property~\eqref{eq:*},
 i.e.~the modulus of $x$ is proportional to the modulus of $\beta^L$, where $L$ is the leading index of the representation.
The answer to such question is negative: One can always find sequences of rational integers $(x_1^{(n)})_{n\geq 0},(x_2^{(n)})_{n\geq 0}$ such that
$x_i^{(n)}\ \to\ +\infty$, as $n$ tends to $\infty$, while $x^{(n)}=x_1^{(n)}+x_2^{(n)}\sqrt5\ {\to}\ 0$. Moreover, since $x^{(n)}\in\Q(\sqrt5)$, $x^{(n)}=\sum_{k\geq -L}a_k^{(n)}(\sqrt5)^{-k}$ implies that
$$
x_1^{(n)} = \sum_{2k\geq -L}a_{2k}^{(n)}5^{-k},\quad
x_2^{(n)} = \sum_{2k+1\geq -L}a_{2k+1}^{(n)}5^{-k}\,.
$$
At the same time, representations of rational integers in base $\gamma=5$ and alphabet $\A=\{-2,\dots,2\}$ are unique up to the identity $0\bullet 4^\omega=1$. Therefore, since $x_i^{(n)}\ \to\ +\infty$, the leading indices of the representations of $x_i^{(n)}$ tend to infinity. Therefore one cannot find a constant $c>0$ so that property~\eqref{eq:*} holds in Example~\ref{ex:sqrt5}.

This leads us to the following definition.

\begin{definition}\label{d:weakgreedy}
Given $\beta\in\C,$ $|\beta|>1,$ $\A\subset\Q(\beta),$ and $c>0$, we say that a sequence $(a_{k})_{k=-L}^{\infty}$ is a weak greedy $(\beta,\A)$-representation for $x$ with respect to $c$ if $x=\sum_{k\geq-L}a_{k}\beta^{-k}$ and $|x|\geq c|\beta|^{L}.$
\end{definition}

\begin{pozn}
Suppose $x=\sum_{k\geq-L}a_{k}\beta^{-k}$. Clearly $x-a_{-L}\beta^{L}=\sum_{k\geq-L+1}a_{k}\beta^{-k}$, if $(a_{k})_{k=-L}^{\infty}$ were not a weak greedy $(\beta,\A)$-representation for $x$ with respect to $c=2^{-1}\cdot\min\{|a_{k}|\},$ then we would have $|x-a_{-L}\beta^{L}|>x$. One might expect that a reasonable algorithm for generating $(\beta,\A)$-representations would be one for which $|x-a_{-L}\beta^{L}|<x$, that is one for which after applying the first step of the algorithm, i.e. determining the first digit, one is left with a smaller/simpler number with which one has to construct a $(\beta,\A)$-representation. This idea is the motivation behind Definition \ref{d:weakgreedy}. Intuitively the quantity $c$ describes how effectively an algorithm determines the leading digit.
\end{pozn}

\begin{lem}\label{l:3}
Let $\beta\in\R$, $|\beta|>1$, be an algebraic number. Assume that there exist $m\in \mathbb{Z}, m>0$ such that $\Q(\beta^m)\neq\Q(\beta)$. Then one cannot find an alphabet $\A\subset\Q(\beta)$ and a constant $c>0$, so that every $x\in\Q(\beta)$ has an eventually periodic weak greedy $(\beta,\A)$-representation with respect to $c$.
\end{lem}

\pf  We prove the lemma by contradiction. Suppose that $\beta$ admits an alphabet $\mathcal{A}$ and a constant  $c>0$  so that every $x\in\Q(\beta)$ has an eventually periodic weak greedy $(\beta,\A)$-representation with respect to $c$.

Since $\beta\in\R$ and $\Q(\beta^m)\neq\Q(\beta)$, $\beta$ is irrational. Therefore one can find two sequences of integers $x_0^{(n)}$ and $x_1^{(n)}$ such that  $x^{(n)}:=x_0^{(n)}+x_1^{(n)}\beta\to 0$ and $|x_i^{(n)}| \to +\infty$ as $n$ tends to infinity.

Denote $\gamma=\beta^m$  and take the eventually periodic representation of $x^{(n)}$  with the leading coefficient $L_n$ satisfying \eqref{eq:*}. Then
$$
x^{(n)} = \sum_{i\geq -L_n}a^{(n)}_i\beta^{-i} =  A^{(n)}_0 + A^{(n)}_1\beta + \cdots  +A^{(n)}_{m-1} \beta^{m-1},
$$
 where
$$
A^{(n)}_j = \sum_{mi+j\geq -L_n} a^{(n)}_{mi+j} \gamma^{-i}  \in \mathbb{Q}(\gamma)\quad \hbox{for} \ j=0,1,\ldots, m-1.
$$
As $x^{(n)}\to 0$, the inequality \eqref{eq:*} implies $L_n \to -\infty$ and thus for every $j=0,1,\ldots, m-1$, $ A^{(n)}_j \to 0$ with $n\to\infty$.

  Note that  $[\Q(\beta):\Q(\gamma)]<+\infty$. In particular, there exist $d\in\N$, $d\geq 2$, such that
$1,\beta,\dots,\beta^{d-1}$ forms a basis for $\Q(\beta)$ over $\Q(\gamma)$.
This means that for $k\geq d$, one can express $\beta^k$  as
 a linear combination of $1,\beta,\dots,\beta^{d-1}$  with coefficients in $\Q(\gamma)$.
 We can therefore write for $x^{(n)}$,
\begin{equation}\label{eq:uni}
x_0^{(n)}+x_1^{(n)}\beta =  \ x^{(n)} = B^{(n)}_0 + B^{(n)}_1\beta + \cdots  +B^{(n)}_{d-1} \beta^{d-1}\,,
\end{equation}
where   $ B^{(n)}_k$ is a linear combination of  at most  $ m-d$  numbers $A^{(n)}_j$  with coefficients from  $\Q(\gamma)$.    Consequently,  $\ B^{(n)}_k \in \mathbb{Q}(\gamma)$  and  $ B^{(n)}_k \to 0$ with $n\to\infty$ for  $k=0,1,\ldots, d-1$.  However, the expression of $x^{(n)}$ in the form~\eqref{eq:uni} is unique, therefore $B^{(n)}_0=x_0^{(n)}$. This leads to contradiction, since by assumption, we have $|x_0^{(n)}|\to  +\infty$.
\pfk

\begin{pozn}
One can extend the statement of Lemma~\ref{l:3} to complex bases $\beta$ which do not generate imaginary quadratic fields. For $\beta\in\C\setminus\R$, the set $\Z+\Z\beta$ is not dense in $\C$. However, $\Z+\Z\beta+\Z\beta^2$ is, if $\beta$ is not quadratic. Therefore one finds a sequence
$x^{(n)}:=x_0^{(n)}+x_1^{(n)}\beta+x_2^{(n)}\beta^2\to 0$ with $x_i^{(n)}\in\Z$ such that $|x_i^{(n)}| \to +\infty$. The rest of the proof remains the same.
Note that exclusion of non-real quadratic bases is not a problem: For such bases we know the answer for the question about weak greedy expansions. If $\beta$ is an algebraic integer, then it is a complex Pisot number, and by Theorem~\ref{t:2} eventually periodic weak greedy expansions can be found. On the other hand, it will be shown in the following proposition, that algebraic non-integers do not possess this property.
\end{pozn}

\begin{prop}\label{p:wg}
Let $\beta\in\C$, $|\beta|>1$, be such that there exists $\mathcal{A}\subset \Q(\beta)$ and $c>0$ such that that every element of $\Q(\beta)$ has an eventually periodic weak greedy $(\beta,\A)$-representation with respect to $c$. Then $\beta$ is an algebraic integer such that for every conjugate $\beta'$ of $\beta$ it holds that $|\beta'|=|\beta|$ or $|\beta'|\leq 1$.
\end{prop}

\pf
By Lemma \ref{l:integer alphabet} we may assume that $\A\subset \Z$. First we show that for any $\beta$ there exists a strictly increasing sequence of integers $(k_n)_{n\geq 0}$ such that
\begin{equation}\label{eq:im}
\lim_{n\to\infty} \frac{\Im \beta^{k_n}}{\beta^{k_n}} = 0\,.
\end{equation}
Denote $r,\varphi$ real numbers such that $\beta=re^{2\pi i\varphi}$. If $\varphi$ is rational, $\varphi=\frac{p}{q}$, then set $k_n=q^n$. Thus $\beta^{k_n}\in\R$ and $\Im \beta^{k_n}=0$. Suppose that $\varphi\notin\Q$. Then $0$ is a limit point of the sequence $n\varphi \mod 1$. Therefore one can find a sequence $(k_n)_{n\geq 0}$ such that $k_n\varphi \mod 1\to 0$. Write
$$
\beta^{k_n} = |\beta|^{k_n} \big(\cos (2\pi k_n\varphi) + i \sin (2\pi k_n\varphi)\big)\,,
$$
whence
$$
\left|\frac{\Im \beta^{k_n}}{\beta^{k_n}}\right| = \left|\frac{\frac1{|\beta|^{k_n}}\Im \beta^{k_n}}{\frac1{|\beta|^{k_n}}\beta^{k_n}}\right| =
\left|\frac{\sin (2\pi k_n\varphi)}{\cos (2\pi k_n\varphi) + i \sin (2\pi k_n\varphi)}\right|\,.
$$
Since $\sin (2\pi k_n\varphi)\to 0$ and $\cos (2\pi k_n\varphi)\to 1$, relation~\eqref{eq:im} is satisfied.

Denote $x_n:=\beta^{k_n}\in\Q(\beta)$, $N:=\lfloor\Re \beta^{k_n} \rfloor\in\Z$, $y_n:=\beta^{k_n}-N \in\Q(\beta)$.
We have $y_n=\Re \beta^{k_n} - \lfloor\Re \beta^{k_n} \rfloor +i\Im \beta^{k_n}$.
Denoting further $b:=\Re \beta^{k_n}- \lfloor\Re \beta^{k_n} \rfloor \in(0,1)$, and $\varepsilon_n:=\frac{\Im \beta^{k_n}}{\beta^{k_n}}$,  we can estimate
$$
|y_n| \leq b + |\Im \beta^{k_n}| \leq 1+ |\varepsilon_n| |\beta^{k_n}|\,.
$$

By assumption, $y_n$ has an eventually periodic $(\beta,\A)$-representation $y_n = \sum_{k\geq -l(y_n)}a_k\beta^{-k}$ with leading index $l(y_n)$ satisfying
$$
\frac{|y_n|}{|\beta^{l(y_n)}|} \geq c\,.
$$
Combining this with the above, we can write $c|\beta^{l(y_n)}| \leq |y_n| \leq 1+|\varepsilon_n| |\beta^{k_n}|$,
which further implies that
$$
c|\beta^{l(y_n)-k_n}| \leq \frac1{|\beta^{k_n}|}+|\varepsilon_n|  \to 0\,.
$$
As a consequence, we have for the sequence of exponents
\begin{equation}\label{eq:indexy}
\lim_{n\to\infty} (l(y_n) - k_n) = -\infty\,.
\end{equation}

We can derive that the integer $N=x_n-y_n$ has an eventually periodic $(\beta,\A)$-representation of the form
\begin{equation}\label{eq:N}
N=\beta^{k_n} - \sum_{k\geq -l(y_n)}a_k\beta^{-k} = \beta^{k_{n}} - a_{l(y_n)}\beta^{l(y_n)} - a_0\beta^0 - \frac{z}{\beta^r(\beta^s-1)}\,,
\end{equation}
where $z=\sum_{i=0}^{rs}b_i\beta^i\in\Z[\beta]$, $r,s\in\Z$, $r\geq0$, and $s$ is the period-length.
Multiplying~\eqref{eq:N} by $\beta^r(\beta^s-1)$, we find that $\beta$ is a zero of a polynomial
$$
P(x) :=x^{k_n}(x^s-1)x^r-Nx^r(x^s-1) - g(x)\,,
$$
where $g\in\Z[x]$ is such that ${\rm st}\cdot deg(g)\leq r s l(y_n)<r s k_n$. Since $P$ is monic, $\beta$ is an algebraic integer.

Suppose that there exists a conjugate $\gamma\neq \beta$ of $\beta$ such that $|\gamma|>1$. Consider the field isomorphism $\sigma$ induced by $\beta\mapsto\gamma$. We apply $\sigma$ to $N$ as given in~\eqref{eq:N},
$$
N=\gamma^{k_n} - a_{l(y_n)}\gamma^{l(y_n)} - a_0\gamma^0 - \frac{\sigma(z)}{\gamma^r(\gamma^s-1)}\,.
$$
Subtracting from~\eqref{eq:N}, we obtain
$$
0=\beta^{k_n}-\gamma^{k_n} - \sum_{k\geq - l(y_n)}a_k(\beta^{-k}-\gamma^{-k})\,,
$$
which then gives
$$
0=\beta^{k_n-l(y_n)}-\gamma^{k_n-l(y_n)} - \sum_{i\geq 0 }c_i(\beta^{-i}-\gamma^{-i})\,,
$$
where $c_i\in\A$. We thus estimate
\begin{equation}\label{eq:bg}
\big|\beta^{k_n-l(y_n)}-\gamma^{k_n-l(y_n)}\big| \leq \max \big\{|a|:a\in\A\big\} \cdot \max \big\{\frac{|\beta|}{|\beta|-1},\frac{|\gamma|}{|\gamma|-1}\big\}\,.
\end{equation}
If $|\gamma|\neq|\beta|$, then without loss of generality $|\beta|>|\gamma|$, and thus the term on the left hand side of~\eqref{eq:bg}
diverges
$$
\lim_{n\to\infty}|\beta|^{k_n-l(y_n)}\big|1-(\tfrac\gamma\beta)^{k_n-l(y_n)}\big| = +\infty
$$
The right hand side of~\eqref{eq:bg} is bounded, which gives a contradiction.
\pfk

\begin{thm}\label{t:Simon}
Given $\beta\in\R$, $|\beta|>1$. Suppose there exists a finite alphabet $\A\subset\Q(\beta)$ and a constant $c>0$ such that every $x\in\Q(\beta)$ has an eventually periodic weak greedy $(\beta,\A)$-representation with respect to $c$. Then $|\beta|$ is a Pisot number or a Salem number.
\end{thm}

To prove Theorem \ref{t:Simon} we make use of the following result due to Ferguson \cite{Ferguson}.

\begin{thm}\label{t:Ferguson}
Suppose that the irreducible polynomial $f(x)\in\Z[x]$ has $m$ roots, at least one real, on the circle $|z| = c$. Then $f(x) = g(x^m)$ where $g(x)$ has no more than one real root on any circle in $\C$.
\end{thm}

\begin{proof}[Proof of Theorem \ref{t:Simon}]
By Proposition~\ref{p:wg}, if $\beta$ satisfies the hypothesis of our theorem then it is an algebraic integer with conjugates $\beta'$ satisfying either $|\beta'|\leq 1$, or $|\beta'| = |\beta|$. If all conjugates lie in the closed unit disc, then $|\beta|$ is a Pisot or Salem number.

Suppose that there exist a conjugate of $\beta$ satisfying $|\beta'|=|\beta|$. Then by Theorem \ref{t:Ferguson}, we have for the minimal polynomial of $\beta$ that $f(x)=g(x^m)$, where $m\geq 2$ is the number of conjugates on the circle of radius $|\beta|$. Moreover, $g$ has at most one real root on every circle in the complex plane. Thus $\beta=\gamma^m$ for $\gamma$ some root of $g$, and since the degree of $g$ is strictly smaller than the degree of $f$, we have $\Q(\beta)\neq\Q(\gamma)$. By Lemma~\ref{l:3} we may conclude that $\beta$ does not allow weak greedy representations for any $c$.
\end{proof}

A simple adaptation of the proof of Theorem \ref{t:Simon} yields the following theorem which can be seen as a more direct generalization of the result of Schmidt.

\begin{coro}\label{t:Schmidt generalisation}
Fix $\beta>1$. Suppose there exists a finite alphabet $\A\subset \Q(\beta)\cap[0,\infty)$ such that every $x\in \Q(\beta)\cap[0,\infty)$ has an eventually periodic $(\beta,\A)$-representation, then $\beta$ is a Pisot or a Salem number.
\end{coro}

The important observation in the proof of Corollary \ref{t:Schmidt generalisation} is that
for a positive real $\beta>1$ and a non-negative alphabet, any representation of a positive number is weak greedy with respect to some $c>0$. More precisely,
if $x\in\Q(\beta)\cap[0,\infty)$ has a $(\beta,\A)$-representation, then one automatically has that $x\geq c\beta^{L},$ where $L$ is the leading digit in the $(\beta,\A)$-representation and $c$ is a constant depending on $\A$.

\section{Numeration systems allowing parallel addition}

In this section we show that when a base $\beta$ and an alphabet $\A$ allow parallel addition of $(\beta,\A)$-representations, then the question of finding eventually periodic representation of every element of $\Q(\beta)$ is simplified.

\begin{definition}
Given a base $\beta\in\C$, $|\beta|>1$, and an alphabet $\A\subset\C$. Denote $\B=\A+\A$. We say that $(\beta,\A)$ allows parallel addition if there exist $t,r\in\N$ and $\Phi:\B^{t+r+1}\to\A$ such that
\begin{itemize}
\item $\Phi(0^{t+r+1})=0$;
\item For every $x=\sum_{k\in\Z}x_k\beta^{-k}$ with $x_k=0$ for $k< L$ for some $L$ and $x_k\in\B$, it holds that
$x=\sum_{k\in\Z}z_k\beta^{-k}$, where $z_k=\Phi(x_{k-t}\cdots x_kx_{k+1}\cdots x_{k+r})\in\A$.
\end{itemize}
\end{definition}

\begin{thm}[\cite{FrPeSv},\cite{FrHePeSv}]
Let $\beta\in\C$, $|\beta|>1$. Then there exists an alphabet $\A\subset\C$ such that $(\beta,\A)$ allows parallel addition if and only if $\beta$ is an algebraic number and $|\beta'|\neq 1$ for every conjugate $\beta'$ of $\beta$. If this is the case, than one can choose a symmetric alphabet of integer digits, $\A=\{-M,\dots,0,1,\dots,M\}$.
\end{thm}

Using the parallel addition, one can easily derive the following statement.

\begin{coro}\label{c:1}
Let $\beta\in\C$, $|\beta|>1$, and let $\A$ be a symmetric alphabet such that $(\beta,\A)$ allows parallel addition. Denote
$$
\begin{aligned}
{\rm Fin}_\A(\beta) &= \big\{\sum_{k\in I} a_k\beta^{-k} : I\subset\Z,\, I\text{ finite},\, a_k\in\A\big\}\,,\\
{\rm Per}_\A(\beta) &= \big\{\sum_{k\geq -L} a_k\beta^{-k} : L\in\Z,\, a_k\in\A,\, a_{-L}a_{-L+1}a_{-L+2}\cdots\text{ eventually periodic}\big\}\,.
\end{aligned}
$$
Then
\begin{enumerate}
\item\label{en:0} ${\rm Fin}_\A(\beta)\subset {\rm Per}_\A(\beta)$;

\item\label{en:1} ${\rm Fin}_\A(\beta)$, ${\rm Per}_\A(\beta)$ are closed under addition and subtraction;

\item\label{en:2} ${\rm Fin}_\A(\beta)$ is closed under multiplication;

\item\label{en:3} ${\rm Fin}_\A(\beta)\cdot{\rm Per}_\A(\beta)\subset {\rm Per}_\A(\beta)$;

\item\label{en:4} ${\rm Fin}_\A(\beta) = \Z[\beta,\beta^{-1}]$.
\end{enumerate}
\end{coro}

\begin{example}
Let us illustrate how the above corollary can be used to find an eventually periodic expansion of the number $1/5$ in base $\beta=\tfrac32$. Such a base is chosen as the most simple example of an algebraic non-integer base.
 The base $\beta=\tfrac32$ allows parallel addition algorithm in the alphabet $\A=\{-2,\dots,2\}$. Addition is done via conversion of digits in $\A+\A$ back into the alphabet $\A$. Note that the ``for'' loop in the following algorithm is supposed to be computed for
 all the indices $j$ at the same time.
 \begin{algorithm}[H]
   \label{paralg}
  \begin{algorithmic}[1]
    \REQUIRE A sequence $(a_j)_{j\in\Z}$ such that $a_j\in\{-3,\dots,3\}$ for all $j\in\Z$ and $a_j\neq$ only for finitely many indices $j.$
    \ENSURE A sequence $(b_j)_{j\in\Z}$ such that $b_j\in\{-2,\dots,2\}$ for all $j\in\Z$,  $c_j\neq$ only for finitely many indices $j$, and $\sum_{j\in\Z}a_j\beta^j=\sum_{j\in\Z}b_j\beta^j$.

	\FOR{$j\in\Z$ in parallel}
	\IF{$1\leq a_j\leq 3$}
		\STATE $q_j:= 1$
	\ELSE
		\STATE $q_j:= 0$
	\ENDIF
	\STATE $c_j:=a_j-3q_j+2q_{j-1}$
	\IF{$-3\leq c_j\leq -1$}
		\STATE $p_j:= -1$
	\ELSE
		\STATE $p_j:= 0$
	\ENDIF
	\STATE $b_j:=a_j-3p_j+2p_{j-1}$
	\ENDFOR
   \end{algorithmic}
  \caption{Parallel conversion of digits in base $\beta=\tfrac32$}
 \end{algorithm}
Using this algorithm, we can sum the elements of ${\rm Per}_\A(\beta)$. In the following, we will use the notation $a_k\dots a_0\bullet a_{-1}\dots$ for the $(\beta,\A)$-representation $\sum_{i\leq k}a_i\beta^i$.
$$
\begin{array}{cccccccccccccc}
&\dots&0&0&0&1&\bullet&1&0&2&1&0&2&\dots\\
+&\dots&0&0&2&2&\bullet&2&-1&-1&2&-1&-1&\dots\\ \hline
a_j&\dots&0&0&2&3&\bullet&3&-1&1&3&-1&1&\dots\\
q_j&\dots&0&0&1&1& &1&0&1&1&0&1&\dots\\
c_j&\dots&0&2&1&2&\bullet&0&1&0&0&1&0&\dots\\
p_j&\dots&0&0&0&0& &0&0&0&0&0&0&\dots\\ \hline
b_j&\dots&0&2&1&2&\bullet&0&1&0&0&1&0&\dots
\end{array}
$$

Let us now find an eventually periodic $(\beta,\A)$-representation of $\tfrac15.$
We have $3^4\equiv 2^4\equiv1\mod 5,$ in particular $3^4-2^4=13\cdot 5.$ Thus we obtain
$$
\frac15=\frac{13}{2^4}\cdot\frac1{(\tfrac32)^4-1}=\frac{13}{2^4}\cdot\sum_{k=1}^\infty\frac1{\beta^{4k}}.
$$
The geometric series itself is already an eventually periodic $(\beta,\A)$-representaion $0\bullet (0001)^\omega$. Furthemore, we have $\tfrac12=1(-1)\bullet$, from which $\tfrac{13}{2^4} = 10(-1) 0(-2)\bullet$ can be computed by the usual grade-school multiplication followed by applying Algorithm~\ref{paralg} for the digit conversion. Altogether we obtain
$$
\tfrac15= \big(10(-1) 0(-2)\bullet\big) \cdot \big(0\bullet(0001)^\omega\big)=1\bullet (0(-1))^\omega
$$
where again the usual multiplication was used, now (luckily) without the need of the digit conversion.
\end{example}

The approach to obtain eventually periodic expansion of $\tfrac15$ in base $\beta=\frac32$, illustrated in the above example, is  generalized in the following theorem.

\begin{thm}\label{t:X}
Let $\beta\in\C$ be an algebraic number of degree $d$, $|\beta|>1$, and let
$a_dx^d-a_{d-1}x^{d-1}-\cdots-a_1x-a_0\in\Z[x]$ be its minimal polynomial.
Suppose that
\begin{enumerate}
\item $|\beta'|\neq1$ for any conjugate $\beta'$ of $\beta$;
\item \label{en:22} $1/a_d\in\Z[\beta,\beta^{-1}]$.
\end{enumerate}
Then there exists a finite alphabet $\A\subset\Z$ such that $\Q(\beta)={\rm Per}_\A(\beta)$.
\end{thm}

\pf
Consider a symmetric integer alphabet $\A$ which allows for $\beta$ parallel addition. Every $x\in\Q(\beta)$ can be written in the form $x=\frac1q(x_0+x_1\beta+\cdots+x_{d-1}\beta^{d-1})$, where $q\in\N$, $x_0,\dots,x_{d-1}\in\Z$.
Thanks to items~\ref{en:3} and~\ref{en:4} from Corollary~\ref{c:1}, it is sufficient to show that $1/q\in{\rm Per}_\A(\beta)$ for every $q\in\N$.
For simplicity of notation, denote $a_d=a$. Suppose first that $q$ and $a$ are coprime. Let $A$ and $\vec{b}$ be defined as follows:
$$
A:=\begin{pmatrix}
a_{d-1} & a_{d-2} & \cdots & a_1 & a_0\\[1mm]
a_d     & 0       & \cdots & 0  & 0\\
0       & a_{d }  & \ddots  &0  & 0\\
\vdots  & \vdots  & \ddots  &\ddots  & \vdots\\
0       &0       & \cdots   &a_d    &0
\end{pmatrix}
\in\Z^{d\times d}\,,\quad
\vec{b}:=\begin{pmatrix}
\beta^{d-1}\\
\beta^{d-2}\\
\vdots\\
\beta\\
1
\end{pmatrix}\,.
$$
As $a\beta^d = a_{d-1}\beta^{d-1}+\cdots+a_1\beta+a_0$,
we have the relation
\begin{equation}\label{eq:trojuhelnik}
A \vec{b}=a\beta \vec{b}\,.
\end{equation}
On the set $\{A^k : k\in\N\}$, define the relation $\sim$ as follows: $A^k\sim A^n$ if $q$ divides every entry of the matrix $A^k-A^n$. Obviously, $\sim$ is an equivalence relation which partitions $\{A^k : k\in\N\}$ into a finite number of equivalence classes. Therefore there exist indices $m+l>m\geq 0$ such that $A^{m+l}-A^m=A^m(A^l-1)=qC$ for a matrix $C\in\Z^{d\times d}$. Since for any index $n$ we have $A^{ln}-I=(A^l-I)(I+A^l+\cdots+A^{(n-1)l})$, we can write
\begin{equation}\label{eq:**}
A^m(A^{ln}-I)=qD\,,\quad \text{ for some }D\in\Z^{d\times d}\,.
\end{equation}

Now recall that $a$ is coprime to $q$, thus there exists $n\in\N$ such that $a^n-1\in q\Z$. For example, take $n=\varphi(q)$, where $\varphi$ is the Euler totient function. This implies that for any $j\in\N$, we have $a^{nj}-1 \in q\Z$. Setting $s=ln$, we obtain that
\begin{equation}\label{eq:***}
A^m(a^{s}I- I)= qE ,\quad \text{ for some }E\in\Z^{d\times d}\,.
\end{equation}
Subtracting~\eqref{eq:***} from~\eqref{eq:**}, we obtain for $s=ln$ that
$$
A^m(A^s-a^sI) = q Z\,,\quad \text{where } Z\in\Z^{d\times d}\,.
$$
Relation~\eqref{eq:trojuhelnik} then implies
\begin{equation}\label{eq:trojuhelnik2}
A^m(A^s-a^sI) \vec{b}=a^{s+m}\beta^m(\beta^s-1) \vec{b} = q Z \,\vec{b}\,.
\end{equation}
The components of the vector $Z\,\vec{b}$ all belong to $\Z[\beta]$. Denoting its last component by $z$, we obtain from~\eqref{eq:trojuhelnik2} that $a^{s+m}\beta^m(\beta^s-1)=qz$, whence we derive
\begin{equation}\label{eq:1-1}
\frac{1}{q} = z\cdot \frac{1}{a^{s+m}}\cdot \frac{1}{\beta^m(\beta^{s}-1)}\,,\quad \text{for some } z\in\Z[\beta]\,.
\end{equation}
The assumptions of the theorem together with items~\ref{en:4},~\ref{en:1}, and~\ref{en:2} of Corollary~\ref{c:1} imply that
\begin{equation}\label{eq:1-3}
z\cdot\frac1{a^{s+m}}\cdot\frac1{\beta^m}\in{\rm Fin}_\A(\beta)\,.
\end{equation}
Moreover,
\begin{equation}\label{eq:1-2}
\frac1{\beta^s-1} = \sum_{i=1}^\infty\frac{1}{\beta^{si}}\in{\rm Per}_\A(\beta)\,.
\end{equation}
By item~\ref{en:3} of Corollary~\ref{c:1}, relations~\eqref{eq:1-1},~\eqref{eq:1-3}, and~\eqref{eq:1-2} imply that $\frac1q\in{\rm Per}_\A(\beta)$, as desired.

It remains to consider the case when $\gcd(a,q)>1$. In this case $q=r\overline{q}$, where $\gcd(a,\overline{q})=1$. According to the above, we have $\frac1{\overline{q}}\in{\rm Per}_\A(\beta)$. Moreover, $a=r\overline{a}$ for some $\overline{a}\in\Z$. Thus $\frac1{r}=\frac{\overline{a}}{a}$. By the assumption~\ref{en:22} of the theorem and by items~\ref{en:2} and~\ref{en:4} of Corollary~\ref{c:1}, we derive that $\frac1r\in{\rm Fin}_\A(\beta)$. By~\ref{en:3} of Corollary~\ref{c:1}, we have $\frac1q=\frac1r\frac1{\overline{q}}\in{\rm Per}_\A(\beta)$.
\pfk

\begin{coro}\label{coro:algint}
Let $\beta$ be an algebraic integer with no conjugate $\beta'$ in modulus equal to 1. Then there exists an alphabet $\A\subset\Z$ such that $\Q(\beta) = {\rm Per}_\A(\beta)$.
\end{coro}

\pf
For an algebraic integer, we have $|a_d|=1$.
\pfk

\begin{coro}\label{coro:rational}
Let $\beta=\frac{p}{q}\in\Q$, $p\perp q$, $|p|>q\geq 1$. Then there exists an alphabet $\A\subset\Z$ such that $\Q(\beta) = {\rm Per}_\A(\beta)$.
\end{coro}

\pf
The minimal polynomial of $\beta$ is $qx-p$, thus it is sufficient to verify that $\frac1q\in\Z[\frac{p}{q},\frac{q}{p}]$. Since $p$ and $q$ are coprime, there exist $x,y\in\Z$ such that $px+qy=1$. Whence we have
$\frac1q = y+x\beta\in\Z[\beta]$.
\pfk

\begin{thm}\label{thm:finite}
Let $\beta$, $|\beta|>1$, be an algebraic number with minimal polynomial $f(x)=\sum_{i=0}^da_ix^i\in\Z[\beta]$, $\gcd(a_0,\dots,a_d)=1$, and let $n\in\N$, $n\geq 2$.

Then $1/n\in\Z[\beta,\beta^{-1}]$  if and only if its prime factorization is
$n=\prod_{1\leq k\leq m} p_k^{\alpha_k}$, $\alpha_k\geq 1$, where each $p_k$ divides $a_i$ for all $i\in\{0,\dots,d\}$ but one (possibly different for each $k$).

\end{thm}
\begin{proof}
First assume that $1/n\in\Zbb$, in particular, that $1/n = w/\beta^k$ for some $w\in\Z[\beta]$ and $k\in\N$. This is equivalent to $\beta^k-nw=0$, i.e.\ $\beta$ is a root of the polynomial $g(x) = x^k-\widetilde g(x)$ with $\widetilde g\in n\Z[x]$. The minimal polynomial $f$ of $\beta$ necessarily divides $g$ in $\Z[x]$.

Consider a prime divisor $p$ of $n$. Then all the coefficients of $g(x)=\sum b_ix^i$ but one are divisible by $p$. If $g=fh$ with $f=\sum a_ix^i,\ g=\sum c_ix^i$, set $j$ and $k$ to be the minimal indices such that $p\nmid a_j$ and $p\nmid c_k$. Then $b_{j+k} = \sum_{l+r=j+k} a_lc_r$ is not divisible by $p.$ Now if we denote $r$ and $s$ the maximal indices for which $p\nmid a_r$ and $p\nmid c_s$, we obtain that $p\nmid a_{r+s}$ and thus $r+s=j+k$ from the uniqueness of the coefficient of $g$ that is not divisible by $p$. The equality $r=j, s=k$ follows from $r\geq j,s\geq k.$

To prove the opposite direction, we show that for every  prime $p$ dividing $a_i$ for all indices but some $j$-th, the number $1/p$ has a finite expansion. The statement then follows from item~\ref{en:4} of Corollary~\ref{c:1}, i.e.\ from the fact that  $\Z[\beta,\beta^{-1}]={\rm Fin}_{\A}(\beta)$. As $a_j$ is coprime to $p$, the equation $a_jx+py=1$ has a solution $(x,y)\in\Z^2$. We have $1/p = a_jp^{-1} x + y$. Since $a_j = -\beta^{-j} \sum_{0\leq i\leq d, i\neq j}a_i\beta^i$, we obtain that
$$
\frac1{p} = y + \frac{x}\beta \sum_{0\leq i\leq d, i\neq j}(a_i/p)\beta^i \in\Z[\beta,\beta^{-1}].
$$
\end{proof}

\section{Comments}

Representations of numbers in  base $\beta$ and alphabet $\A$ generated by various algorithms are extensively studied, mainly for their arithmetic properties. Most common example are the greedy expansions in base $\beta>1$, which use the canonical alphabet $\{a\in\Z:0\leq a<\beta\}$. One considers the set ${\rm Per}(\beta)$ of numbers with eventually periodic greedy $\beta$-expansions~\cite{Parry,Bertrand,Schmidt}. The set ${\rm Fin}(\beta)$ of numbers whose greedy $\beta$-expansions have only finitely many non-zero digits
$$
{\rm Fin}(\beta)=\{x\geq 0:\text{ the greedy expansion of $x$ is finite}\}
$$
has been studied for example in~\cite{FruSo,AkiPisotGreedy}. In~\cite{BuFrGaKr}, the set $\Z_\beta$ of numbers whose greedy $\beta$-expansions use only non-negative powers of the base has been introduced.

The subject of this paper is representations in base $\beta\in\C$ with digits in an alphabet $\A$ not necessarily issued out of a given algorithm. Instead of ${\rm Per}(\beta)$, ${\rm Fin}(\beta)$, we can consider ${\rm Per}_\A(\beta)$, ${\rm Fin}_\A(\beta)$ as defined in
Corollary~\ref{c:1}. Numbers with only non-negative powers of $\beta$ in their $(\beta,\A)$-representation can be seen as polynomials in $\beta$ with coefficients in $\A$, i.e. form the set $\A[\beta]$.

Comparison of the behaviour of these sets is surprising. While $\Z_\beta$ is never closed under addition except for integer $\beta>1$, Akiyama et al.~\cite{AkiThuZai} show that there exists a finite set $\A\subset\Z$ such that $\A[\beta]=\Z[\beta]$ if and only if $\beta$ is an algebraic number such that $\beta$ and its conjugates lie either all on the unit circle, or all outside the unit circle.

The so-called (F) property has been studied in~\cite{FruSo}. They showed for the greedy $\beta$-expansions that the equality ${\rm Fin}(\beta)=\Z[\beta,\beta^{-1}]\cap[0,+\infty)$ forces $\beta$ to be a Pisot number. On the other hand, allowing any representation, we can derive from~\cite{FrPeSv} that there exists an alphabet $\A\subset\Z$ such that  ${\rm Fin}_\A(\beta)=\Z[\beta,\beta^{-1}]$ for any algebraic number $\beta$ without a conjugate on the unit circle. It is not known whether this statement can be reversed.

Our main tool in showing that certain classes of numbers satisfy ${\rm Per}_\A(\beta)=\Q(\beta)$ is parallel addition in the numeration system $(\beta,\A)$, only known to be possible for algebraic bases with no conjugate on the unit circle. Our computational experiments suggest that every such number $\beta$ admits an alphabet $\A$ such that ${\rm Per}_\A(\beta)=\Q(\beta)$.

It remains an open question whether this equality can hold for some algebraic number $\beta$ having a conjugate of modulus one.
It can be commented that Salem numbers belong to this class. Already when considering greedy $\beta$-expansions generated by the map $x\mapsto \beta x\mod 1$, only very restricted results are known. Schmidt~\cite{Schmidt} conjectured, that for Salem numbers $\beta$, the orbit of any rational $x\in[0,1)$ is eventually periodic, which amounts to stating that equality ${\rm Per}(\beta)=\Q(\beta)\cap[0,+\infty)$ holds. When trying to prove this conjecture, Boyd in~\cite{Boyd} has shown that
for any Salem number $\beta$ of degree $4$ the orbit of $1$ under the map $x\mapsto \beta x\mod 1$ is eventually periodic. On the other hand, in~\cite{Boyd96} the same author provides data which put the Schmidt's conjecture in doubt.

However, even if ${\rm Per}(\beta)=\Q(\beta)\cap[0,+\infty)$ were not true for Salem numbers $\beta$ when considering the greedy $\beta$-expansions, one may still ask about possibility of finding an alphabet of digits $\A$ so that ${\rm Per}_\A(\beta)=\Q(\beta)$ is valid.

\section*{Acknowledgements}
This work was supported by the Czech Science Foundation, grant No.\ 13-03538S. We also acknowledge financial support of
the Grant Agency of the Czech Technical University in Prague, grant No.\ SGS14/205/OHK4/3T/14.

\bibliographystyle{plain}
\IfFileExists{biblio.bib}{\bibliography{Biblio}}{\bibliography{../!bibliography/Biblio}}

\end{document}